\documentclass[11pt,reqno]{amsart}
\usepackage[all]{xy}
\usepackage{amssymb}
\usepackage{amsthm}
\usepackage[normalem]{ulem}
\usepackage{amsmath,mathtools}
\usepackage{amscd,enumitem}
\usepackage{verbatim}
\usepackage{eurosym}
\usepackage{float}
\usepackage{color}
\usepackage{url}
\usepackage{dcolumn}
\usepackage[mathscr]{eucal}
\usepackage[all]{xy}
\usepackage{bbm}
\usepackage[textheight=8.5in, textwidth=6.7in]{geometry}
\newtheorem*{conj*}{Conjecture}
\newtheorem{theorem}{Theorem}[section]

\theoremstyle{definition}
\newtheorem*{remark}{Remark}
\theoremstyle{plain}

\newtheorem{lemma}[theorem]{Lemma}

\newtheorem{corollary}[theorem]{Corollary}

\newcommand{\Z}{\mathbb{Z}}
\newcommand{\Q}{\mathbb{Q}}

\newcommand{\N}{\mathbb{N}}

\newcommand{\SL}{\operatorname{SL}}
\newcommand{\SU}{\operatorname{SU}}

\newcommand{\re}[1]{\text{Re}\(#1\)}

\renewcommand{\pmod}[1]{\,\,({\rm mod}\,\,{#1})}

\numberwithin{equation}{section}

\newtheoremstyle{example}
  {\topsep}   
  {\topsep}   
  {\normalfont}  
  {0pt}       
  {\bfseries} 
  {.}         
  {5pt plus 1pt minus 1pt} 
  {}          
\theoremstyle{example}

\def\({\left(}
\def\){\right)}

\def\lp{\left(}
\def\rp{\right)}

\usepackage{centernot}

\begin{document}
\title[Vafa--Witten invariants of $K3$ surfaces]{Exact formulae and Tur\'{a}n inequalities for Vafa--Witten invariants of $K3$ surfaces}

\author{Daniel R. Johnston}
\address{School of Science, The University of New South Wales, Canberra, Australia}
\email{daniel.johnston@adfa.edu.au}

\author{Joshua Males}
\address{Department of Mathematics, Machray Hall, University of Manitoba, Winnipeg,
	Canada}
\email{joshua.males@umanitoba.ca}

\thanks{The research conducted by the second author for this paper is supported by the Pacific Institute for the Mathematical Sciences (PIMS). The research and findings may not reflect those of the Institute. \\
Competing interests: The authors declare none}


\begin{abstract} 
We consider three different families of Vafa--Witten invariants of K3 surfaces. In each case, the partition function that encodes the Vafa--Witten invariants is given by combinations of twisted Dedekind $\eta$-functions. By utilising known properties of these $\eta$-functions, we obtain exact formulae for each of the invariants and prove that they asymptotically satisfy all higher-order Tur\'{a}n inequalities.
\end{abstract}

\maketitle
\section{Introduction and statement of results}

In 1994, Vafa and Witten \cite{VW} introduced invariants that count the number of solutions to the $N=4$ supersymmetric Yang--Mills equations over four dimensional spaces, now called Vafa--Witten invariants. Given a surface and group action, many results in the literature focus on providing the generating function $Z$ of the Vafa--Witten invariants (also known as the partition function).  These generating functions are often linear combinations of modular objects. Thus, with the tools of modular forms at our disposal, we are in a position to obtain precise information on the Vafa--Witten invariants in question. One of the most natural family of surfaces, which we consider throughout, are the $K3$ surfaces, and there have been several recent results for the generating function of Vafa--Witten invariants in this case \cite{JK,TT,Thomas}. 

In what follows, we let
\begin{equation}\label{etadef}
    \eta(q)\coloneqq q^{\frac{1}{24}}\prod_{n=1}^\infty(1-q^n)
\end{equation}
be the usual Dedekind $\eta$-function, where $q\coloneqq e^{2\pi i \tau}$ with $\tau \in \mathbb{H}$. 

\subsection{Formulae for Vafa--Witten invariants}
Tanaka and Thomas proved the following formula for the generating function of $\SU(r)$ Vafa--Witten invariants on $K3$ surfaces, thereby proving a conjecture of Vafa and Witten (here and throughout, $r \in \mathbb{N}$).
\begin{theorem}[Theorem 5.24 of \cite{TT}]\label{vafaw1}
For a $K3$ surface, the partition function of rank $r$ trivial determinant $\SU(r)$ Vafa--Witten invariants is
\begin{align*}
    Z^{SU(r)}(q) \coloneqq \sum_{n \in \mathbb{Z}}\alpha_{1,r}(n) q^n = \sum_{d \mid r} \frac{d}{r^2} q^r \sum_{j=0}^{d-1} \eta\left(\zeta_d^j q^{\frac{r}{d^2}} \right)^{-24},
\end{align*}
where $\zeta_d = e^{\frac{2\pi i}{d}}$.
\end{theorem}

Moreover, in \cite[Theorem 5.24]{TT} it is shown that the same generating function also gives the rank $r$, trivial determinant, weighted Euler characteristic Vafa--Witten invariants, and so our subsequent results also apply to these invariants.

In \cite{JK}, Jiang and Kool determined the generating function for $SU(p)/\mathbb{Z}_p$ Vafa--Witten invariants for prime $p$. To describe the result, fix a Chern class $c_1 \in H^2(S,\mathbb{Z})$, let $\mu_p$ be the cyclic group of order $p$, and let 
\begin{align*}
    \delta_{a,b} \coloneqq \begin{cases} 1 &\text{ if } a-b \in pH^2(S,\mathbb{Z}), \\
    0 &\text{otherwise}.
   \end{cases}
\end{align*}

\begin{theorem}[Theorem 1.4 of \cite{JK}]\label{vafaw2}
For a $K3$ surface $S$, prime $p$, generic polarisation, and $c_1 \in H^2(S,\mathbb{Z})$ algebraic, we have the partition function
\begin{align*}
   Z_{c_1}^{SU(p)/{\mathbb{Z}_p}} (q) =  \sum_{w \in H^2(S,\mu_p)} e^{\frac{2\pi i (w\cdot c_1)}{p}} Z_w(q), 
\end{align*}
where $w \cdot c_1$ is the natural inner product on $H^2(S,\mathbb{Z})$ (see \cite[pages 45 and 46]{VW}) and
\begin{align}\label{gen2}
    Z_w(q) \coloneqq \sum_{n \in \Z}\alpha_{2,p}(w;n) q^\frac{n}{p} = \frac{\delta_{w,0}}{p^2} \eta\left(q^p\right)^{-24} + \frac{1}{p}\sum_{j=0}^{p-1} e^{-\frac{\pi ij w^2}{p}}   \eta\left(\zeta_p^j q^{\frac{1}{p}}\right)^{-24}.
\end{align}
\end{theorem}

We also consider the twisted Vafa--Witten invariants for smooth projective $K3$ surfaces as described by Jiang (see also \cite[Definition 5.13]{Jiang}).

\begin{theorem}[Theorem 6.36 of \cite{Jiang}]\label{vafaw3}
Let $S$ be a smooth projective $K3$ surface with Picard number $\rho(S)$ and $p$ prime. Then we have the partition function
\begin{align}\label{gen3}
    Z_{tw}^{SU(p)/\mathbb{Z}_p}(q) \coloneqq \sum_{n \in\Z}\alpha_{3,p}(n) q^\frac{n}{p} = \frac{q^p}{p^2} \eta\left(q^p\right)^{-24} + q^{p}\left( p^{21}  \eta\left(q^{\frac{1}{p}}\right)^{-24} +p^{\rho(S) -1} \sum_{j=1}^{p-1} \eta\left(\zeta_p^j q^{\frac{1}{p}}\right)^{-24}  \right).
\end{align}
\end{theorem}
\begin{remark}
    Note that the generating functions in Theorems \ref{vafaw2} and \ref{vafaw3} consist of $\frac{n}{p}$th powers of $q$ unlike the generating function in Theorem \ref{vafaw1}, which only contains integer powers of $q$. However, in special cases, the non-integral powers of $q$ in \eqref{gen2} and \eqref{gen3} vanish. See Corollaries \ref{alpha2cor} and \ref{alpha3cor} for more information.
\end{remark}
 
To obtain an exact formula for each of the Vafa--Witten invariants discussed above, we first obtain an exact formula for the coefficients of $\eta(q)^{-24}$. This is achieved using a result of Zuckerman (see Section \ref{Sec: Zuc}), which relates the coefficients of a weakly holomorphic form to the principal part of the form at each cusp. 

\begin{lemma}\label{anlem}
    Let $a(n)$ denote the $n^{th}$ Fourier coefficient of $\eta(q)^{-24}$. Then, $a(n)=0$ for $n<-1$, $a(-1)=1$ and  $a(0)=24$. For $n>0$ we have 
    \begin{equation*}
        a(n)=\frac{2\pi}{n^{\frac{13}{2}}}\sum_{k=1}^\infty\frac{1}{k}\sum_{\substack{0\leq h<k\\\gcd(h,k)=1}}\omega_{h,k}^{24}e^{-\frac{2\pi i(n+1)h}{k}}I_{13}\left(\frac{4\pi}{k}\sqrt{n}\right).
    \end{equation*}
    Moreover,
    \begin{align}
        a(n)&=\frac{2\pi}{n^{\frac{13}{2}}}I_{13}\left(4\pi\sqrt{n}\right)\left(1+O\left(\frac{1}{e^{2\pi\sqrt{n}}}\right)\right) \sim\frac{2\pi}{n^{\frac{13}{2}}}\frac{e^{4\pi\sqrt{n}}}{\sqrt{8\pi^2\sqrt{n}}} \label{aasym1},
    \end{align}
    where $I_{13}$ is the usual $I$-Bessel function.
\end{lemma}

We then prove the following theorem, which combined with Lemma \ref{anlem} gives exact formulae for the Vafa--Witten invariants.

\begin{theorem}\label{thm: main exact}
    For all $n\in\Z$, we have
    \begin{align}\label{exactalpha1}
        \alpha_{1,r}(n) = \sum_{s\mid\left(n-r,r\right)}\frac{1}{s^2}a\left(\frac{(n-r)r}{s^2}\right),
    \end{align}
    along with
    \begin{align}\label{alpha2exact}
        \alpha_{2,p}(w;n)  =\frac{1}{p}\sum_{j=0}^{p-1}e^{\frac{\pi i j(2n-w^2)}{p}}a(n)+\begin{cases}
        \frac{\delta_{w,0}}{p^2}a\left(\frac{n}{p^2}\right)&(p^2 \mid n)\\
        0&\text{(otherwise)}
        \end{cases},
    \end{align}
    and
    \begin{align}\label{alpha3exact}
        \alpha_{3,p}(n) = p^{21}a(n-p^2)
        +
        \begin{cases}
            \frac{1}{p^2}a\left(\frac{n-p^2}{p^2}\right)&\text{ if } (p^2\mid n)\\
            0&\text{(otherwise)}
        \end{cases}
        +
        \begin{cases}
            p^{\rho(S)-1}(r-1)a(n-p^2),& \text{ if }(p\mid n)\\
            -p^{\rho(S)-1}a(n-p^2),&\text{ if }(p\nmid n)
        \end{cases}
        .
    \end{align}
\end{theorem}

From these exact formulae, we also deduce the asymptotic behaviours of the invariants using \eqref{aasym1}.

\begin{theorem}\label{thm: main asymp}
As $n \to \infty$ we have
\begin{align}\label{alpha1asym}
     \alpha_{1,r}(n)=\frac{2\pi}{(rn)^{\frac{13}{2}}}I_{13}\left(4\pi\sqrt{rn}\right)\left(1+O_r\left(\frac{1}{e^{2\pi\sqrt{rn}}}\right)\right).
\end{align}
For $\alpha_{2,p}(n)$, we have
\begin{equation}\label{alpha2asym}
    \alpha_{2,p}(n)= \begin{cases} \frac{2\pi}{n^{\frac{13}{2}}} I_{13}(4\pi\sqrt{n})\left(1+O_r\left(\frac{1}{e^{2\pi\sqrt{n}}}\right)\right) &\text{ if } 2p \mid 2n-w^2,\\
      \frac{2\pi \delta_{w,0} p^{\frac{9}{2}}}{n^{\frac{13}{2}}}I_{13}\left(\frac{4\pi}{p}\sqrt{n}\right)\left(1+O_p\left(\frac{1}{e^{\frac{2\pi}{p}\sqrt{n}}}\right)\right) &\text{ if } 2p \nmid 2n - w^2\ \text{and}\ p^2 \mid n, \\
    0 &\text{ otherwise}.
    \end{cases}
\end{equation}


Finally,
\begin{align}\label{alpha3asym}
    \alpha_{3,p}(n) =\frac{2\pi c_p(n)}{n^{\frac{13}{2}}}I_{13}\left(4\pi\sqrt{n}\right)\left(1+O\left(\frac{1}{e^{2\pi\sqrt{n}}}\right)\right),
\end{align}
where 
\begin{equation*}
    c_p(n)=
    \begin{cases}
        p^{21}+p^{\rho(S)-1}(p-1),&(p\mid n),\\
        p^{21}-p^{\rho(S)-1},&(p\nmid n).
    \end{cases}
\end{equation*}

\end{theorem}

The case of $\alpha_{2,p}$ is a little complicated in the statement of the theorem, which is explained by the fact that the sum over $j$ can vanish or not as $n$ varies, yielding a non-uniform asymptotic.

\subsection{Tur\'{a}n inequalities} The Tur\'{a}n inequalities for functions in the real entire Laguerre--P\'{o}lya class are intricately linked to the Riemann hypothesis \cite{Dimitrov, Sz} and have seen renewed interest in recent years. In fact, the Riemann hypothesis is true if and only if the Riemann Xi function lies in the Laguerre--P\'{o}lya class, and a necessary condition is that the Maclaurin coefficients satisfy Tur\'{a}n inequalities of all orders \cite{Dimitrov,SchurPolya}.

The second order Tur\'{a}n inequality for a sequence $\{a_n\}_{n \geq 0}$ is known as log-concavity, and is satisfied if
\begin{align*}
    a_n^2 \geq a_{n-1}a_{n+1}
\end{align*}
for all $n \geq 1$. Log-concavity for modular objects is well-studied in the literature, for example \cite{BJMR,CCM,CraigPun,DeSalvo,Males}.
The higher-order Tur\'{a}n inequalities are linked to the Jensen polynomial associated to a sequence of real numbers $\{\alpha(n)\}$, given by
\begin{align*}
    J_{\alpha}^{d,n}(X) \coloneqq \sum_{j=0}^{d-1} \binom{d}{j} \alpha(n+j) X^j.
\end{align*}
Consider a polynomial with real coefficients, $f(x) = x^n + a_{n-1}x^{n-1} + \dots + a_1x+a_0$. Then a classical result of Hermite states that the polynomial $f$ is hyperbolic if and only if the Hankel matrix of $f$ is positive definite.  In turn, this gives a set of inequalities on the coefficients of $f$ as all $k\times k$ minors of the Hankel matrix must be positive definite (see e.g. the introduction of \cite{CraigPun} for an explicit description). When applied to $J_{\alpha}^{d,n}$, Hermite's theorem gives a set of inequalities associated to the sequence $\{\alpha(n)\}_{n \geq 0}$ and are known as the order $k$ Tur\'{a}n inequalities\footnote{One also allows the equality case.}. Thus the higher order (strict) Tur\'{a}n inequalities for $1\leq k \leq d$ are satisfied for $\{\alpha(n+j)\}_{n\geq 0}$ if and only if $J_{\alpha}^{d,n}$ is hyperbolic. 

In \cite{GORZ}, the authors proved new results on the (asymptotic) hyperbolicity of Jensen polynomials attached to a vast array of sequences, including those that are coefficients of weakly holomorphic modular forms. Following this, several other recent papers have made progress in this area, for example \cite{LW}, \cite{Xi} and \cite{CraigPun}.

Our result in this direction readily follows from results of \cite{GORZ} and we obtain the following which, to the best of the authors' knowledge, is the first description of higher-order Tur\'{a}n inequalities for Vafa--Witten invariants. The case for $\alpha_{2,p}$ is slightly more complicated than the others, as it depends inherently on the choice of $w$.
\begin{theorem}\label{Thm: main asymp Turan}
     For fixed $r$ and $p$, as $n \to \infty$ the Vafa--Witten invariants $\alpha_{1,r}(n)$ and $\alpha_{3,p}(pn)$ each satisfy all order $k$ Tur\'{a}n inequalities. If $\gcd(2p,w^2) = 1$ then $\alpha_{2,p}(w;n)$ satisfies all order $k$ Tur\'{a}n inequalities. If $w^2 =2pm$ for some $m \in \mathbb{Z}$, then $\alpha_{2,r}(w;pn)$ satisfies all order $k$ Tur\'{a}n inequalities. Moreover, $\alpha_{2,p}(w;p^2n)$ satisfies all order $k$ Tur\'{a}n inequalities.
\end{theorem}

Since the Jensen polynomial attached to the each of the sequences $\alpha_{1,r}(n)$, $\alpha_{2,p}(pn)$ and $\alpha_{3,p}(pn)$ (replacing $\alpha_{2,p}(n)$ by $\alpha_{2,p}(pn)$ or $\alpha_{2,p}(p^2n)$ as necessary) for fixed $r$ and $p$ are all eventually hyperbolic, in Corollary \ref{Cor: LP} we obtain several examples of generating functions which can be regarded as new examples of functions in the shifted Laguerre--P\'{o}lya class, as in recent work of Wagner \cite{Wagner}. To the best of the authors' knowledge, such functions have so far only previously been studied in the context of number theory and combinatorics. That our functions lie in the shifted Laguerre--P\'{o}lya class is a direct consequence of the generating function of $\alpha_{j,r}(n)$ arising from specialisations of (twists of) Dedekind $\eta$-functions. A natural question to pose is whether further examples of topological invariants define objects in the same class of functions, and therefore what are the implications of general properties of this class of function to the topological spaces at hand, motivating further investigation of the situation in \cite{Wagner}.

\section{Preliminaries}\label{sectprelim}
In this section we collect several preliminary results required for the rest of the paper.
\subsection{An exact formula of Zuckerman}\label{Sec: Zuc}
We require a powerful result of Zuckerman \cite{Zuckerman}, that builds on work of Rademacher \cite{RademacherExact}. Zuckerman obtained exact formulae for Fourier coefficients of weakly holomorphic modular forms of arbitrary non-positive weight on finite index subgroups of $\SL_2(\Z)$ in terms of the cusps of the underlying subgroup and the principal parts of the form at each cusp. Let $\tau \in \mathbb{H}$ and let $F$ be a weakly holomorphic modular form of weight $\kappa \leq 0$ with transformation law
$$F(\gamma \tau) = \chi(\gamma) (-i(c \tau + d))^{\kappa} F(\tau),$$
for all $\gamma = \left( \begin{smallmatrix} a & b \\ c & d \end{smallmatrix} \right)$ in some finite index subgroup of $\SL_2(\Z)$. The transformation law can be expressed equivalently in terms of the cusp $\frac{h}{k} \in\Q$. Let $h'$ be defined by the congruence $hh' \equiv -1 \pmod k$. Letting $\tau = \frac{h'}{k} + \frac{i}{kz}$ and choosing $\gamma=\gamma_{h,k}  := \left( \begin{smallmatrix} h & \beta \\ k & -h' \end{smallmatrix} \right) \in \textnormal{SL}_2(\mathbb{Z})$, we obtain the equivalent transformation law
\begin{equation}\label{maintranseq}
    F\lp \frac{h}{k}+\frac{iz}{k} \rp = \chi(\gamma_{h,k}) z^{-\kappa}   F\lp \frac{h'}{k}+\frac{i}{kz} \rp.
\end{equation}

Let $F$ have the Fourier expansion at $i\infty$ given by
\begin{align*}
	F(\tau) = \sum_{n \gg -\infty} a(n) q^{n + \alpha}
\end{align*}
and Fourier expansions at each rational number $0 \leq \frac{h}{k} < 1$ given by
\begin{align*}
	F|_{\kappa}\gamma_{h,k}(\tau) = \sum_{n \gg -\infty} a_{h,k}(n) q^{\frac{n + \alpha_{h,k}}{c_{k}}}.
\end{align*}

In this framework, Zuckerman's result \cite[Theorem 1]{Zuckerman} (see also \cite{BCMO}) reads as follows.
\begin{theorem}\label{Thm: Zuckerman}
	Assume the notation and hypotheses above. If $n + \alpha > 0,$ then
	\begin{align*}
		&a(n) =  2\pi (n+\alpha)^{\frac{\kappa-1}{2}} \sum_{k=1}^\infty \dfrac{1}{k} \sum_{\substack{0 \leq h < k \\ \gcd(h,k) = 1}}\chi(\gamma_{h,k}) e^{- \frac{2\pi i (n+\alpha) h}{k}} 
		\\ 
		&\ \times \sum_{m+\alpha_{h,k} \leq 0} a_{h,k}(m) e^{ \frac{2\pi i}{k c_{k}} (m + \alpha_{h,k}) h' } \left( \dfrac{\lvert m +\alpha_{h,k} \rvert}{c_{k}} \right)^{ \frac{1 - \kappa}{2}} I_{-\kappa+1}\left( \dfrac{4\pi}{k} \sqrt{\dfrac{(n + \alpha)\lvert m +\alpha_{h,k} \rvert}{c_{k}}} \right),
	\end{align*}
	where $I_{-\kappa+1}$ is the usual $I$-Bessel function.
\end{theorem}

\subsection{Transformation laws of $q$-Pochhamer symbols}
For coprime $h,k\in\N$ let
\begin{align*}
\omega_{h,k}:=\exp(\pi i \cdot s(h,k)),
\end{align*}
with the Dedekind sum
\begin{align*}
s(h,k):=\sum_{\mu \pmod k} \left(\left(\frac{\mu}{k}\right)\right)\left(\left(\frac{h\mu}{k}\right)\right).
\end{align*}
Here,
\begin{align*}
((x)):=\begin{cases} x-\lfloor x \rfloor-\frac{1}{2} & \text{if} \ x\in \mathbb R \setminus \mathbb Z, \\ 0 & \text{if} \ x \in \mathbb Z. \end{cases}
\end{align*}

Let $(a;q)_\infty$ be the usual $q$-Pochhammer symbol and $q \coloneqq e^{\frac{2\pi i}{k} (h+iz)}$ for $z\in \mathbb{C}$ with $\re{z}>0$. The classical modular transformation law for the Dedekind $\eta$-function (see e.g. \cite[Theorem 3.4]{apostol}) is
\begin{equation*}
    \eta\left(\frac{a\tau+b}{c\tau+d}\right)=\epsilon(a,b,c,d)(-i(c\tau+d))^{\frac{1}{2}}\eta(\tau),
\end{equation*}
where
\begin{equation*}
    \epsilon(a,b,c,d)=e^{i\pi\left(\frac{a+d}{12c}+s(-d,c)\right)}
\end{equation*}
if $c>0$. Thus, if $(a,b,c,d)=(h,\beta,k,-h')$, we get
\begin{equation*}
    \eta\left(\frac{h'}{k} + \frac{i}{kz} \right)=e^{i\pi\left(\frac{h-h'}{12k}+s(h',k)\right)}z^{-\frac{1}{2}}\eta\left(\frac{h}{k} +\frac{iz}{k} \right)
\end{equation*}
so that
\begin{equation}\label{etatrans}
    \frac{1}{ \eta\left(\frac{h'}{k} + \frac{i}{kz} \right)}=e^{i\pi\left(\frac{h'-h}{12k}+s(h,k)\right)}z^{\frac{1}{2}}\frac{1}{\eta\left(\frac{h}{k} +\frac{iz}{k} \right)}.
\end{equation}
Here we have used a standard result regarding the Dedekind sum (\cite[Theorem 3.6b]{apostol}), which implies that $s(h,k)=-s(h',k)$.

\subsection{Hyperbolicity of Jensen polynomials}
In \cite{GORZ} it was shown that for certain classes of functions, the Jensen polynomials of the coefficients tend to the order $d$ Hermite polynomials $H_d(X)$, in turn implying asymptotic hyperbolicity. 

\begin{theorem}[Theorem 3 of \cite{GORZ}]\label{Thm: Gorz}
Let $\{ \alpha(n)\}$, $\{A(n)\}$, and $\{\delta(n)\}$ be three sequences of positive real numbers with $\delta(n)$ tending to $0$ and satisfying
\begin{align*}
\log\left( \frac{\alpha(n+j)}{\alpha(n)} \right) = A(n)j - \delta(n)^2j^2+\sum_{i=3}^{d} g_i(n)j^i +o(\delta(n)^d)
\end{align*}
as $n \to \infty$, where each $g_i(n) = o(\delta(n)^i)$ for each $3\leq i \leq d$. Then the renormalized Jensen polynomials 
\begin{align*}
    \widehat{J}_{\alpha}^{d,n}(X) \coloneqq \frac{\delta(n)^{-d}}{\alpha(n)} J_{\alpha}^{d,n}\left( \frac{\delta(n)X-1}{\exp(A(n))}\right)
\end{align*}
satisfy
\begin{align*}
    \lim_{n \to \infty} \widehat{J}_{\alpha}^{d,n}(X) = H_d(X),
\end{align*}
uniformly for $X$ in any compact subset of $\mathbb{R}$. Moreover, this implies that the Jensen polynomials $J_{\alpha}^{d,n}$ are each hyperbolic for all but finitely many $n$.
\end{theorem}

\section{Vafa--Witten invariants}
The goal of this section is to obtain the claimed exact formulae and asymptotics for $\alpha_{1,r}(n)$, $\alpha_{2,p}(w;n)$ and $\alpha_{3,p}(n)$. We begin with a short subsection computing an exact formula for the coefficients $a(n)$.
\subsection{Formula for the coefficients of $\eta(q)^{-24}$}
We begin by proving Lemma \ref{anlem}, which gives an exact formula for the $n^{\text{th}}$ coefficient of $\eta(q)^{-24}=\sum_{n\in\Z}a(n)q^n$. Note that in \cite{IJT}, a longer and more general argument, not using Zuckerman's formula, is used to compute exact formulae for fractional powers of $\eta$. 

\begin{proof}[Proof of Lemma \ref{anlem}]
By the definition \eqref{etadef} of $\eta(q)$, we have $a(n)=0$ for $n<-1$, $a(-1)=1$ and $a(0)=24$. For $n>0$ we apply Zuckerman's formula. So, let $F(q)=\eta(q)^{-24}$. For the setup of Zuckerman's formula, we write
\begin{equation*}
    \tau=\frac{h'}{k}+\frac{i}{kz},\quad q=e^{2\pi i\tau},\quad\text{and}\quad q_1=e^{2\pi i(\gamma_{h,k}(\tau))},
\end{equation*}
where $\gamma_{h,k}= \left( \begin{smallmatrix} h & \beta \\ k & -h' \end{smallmatrix} \right)$ as in Section \ref{Sec: Zuc}.

Now, using \eqref{etatrans}, we have
\begin{equation}\label{etatrans24}
    F(q_1)=e^{24\pi i\left(\frac{h'-h}{12k}+s(h,k)\right)}z^{12}F(q).
\end{equation}

In terms of Zuckerman's formula, $\alpha=0$ and from \eqref{etatrans24} we see that $\kappa=-12$ and
\begin{equation*}
    \chi(\gamma_{h,k})=\omega_{h,k}^{24} e^{2\pi i\left(\frac{h'-h}{k}\right)}.
\end{equation*}
Next, by \eqref{etatrans24},
\begin{equation*}
    F|_{\kappa}\gamma_{h,k}(\tau)=F(\tau)=\sum_{n\in\Z}a(n)q^n.    
\end{equation*}
Thus, we can set $c_k=1$, $\alpha_{h,k}=\alpha=0$ and $a_{h,k}(n)=a(n)$. Moreover, the only term in the principal part of $F(\tau)$ is $q^{-1}$ with coefficient $a(-1)=1$. Putting all of this information into Theorem \ref{Thm: Zuckerman} gives, for $n>0$,
\begin{align}
    a(n)&=\frac{2\pi}{n^{\frac{13}{2}}}\sum_{k=1}^\infty\frac{1}{k}\sum_{\substack{0\leq h<k\\\gcd(h,k)=1}}\omega_{h,k}^{24}e^{2\pi i\left(\frac{h'-h}{k}\right)}e^{-\frac{2\pi inh}{k}}e^{-\frac{2\pi i h'}{k}}I_{13}\left(\frac{4\pi}{k}\sqrt{n}\right)\notag\\
    &=\frac{2\pi}{n^{\frac{13}{2}}}\sum_{k=1}^\infty\frac{1}{k}\sum_{\substack{0\leq h<k\\\gcd(h,k)=1}}\omega_{h,k}^{24}e^{-\frac{2\pi i(n+1)h}{k}}I_{13}\left(\frac{4\pi}{k}\sqrt{n}\right),\label{exactan}
\end{align} 
as desired.

For the asymptotic expression, we note that for any $\nu\geq 0$ (see e.g. \cite[(4.12.7)]{andrews})
\begin{align}\label{I-bessel asymp}
    I_\nu(x) = \frac{e^{x}}{\sqrt{2\pi x}} \left(1+O\left(x^{-1}\right)\right)
\end{align}
meaning the $k=1$ term dominates in \eqref{exactan}.
\end{proof}

We are now in a position to compute the exact formulae and asymptotics for $\alpha_{1,r}(n)$, $\alpha_{2,p}(w;n)$, and $\alpha_{3,p}(n)$ (Theorems \ref{thm: main exact} and \ref{thm: main asymp}).

\subsection{Exact formula for $\alpha_{1,r}(n)$}\label{alpha1sect}
In this subsection we prove the exact formula \eqref{exactalpha1} for $\alpha_{1,r}(n)$ and the corresponding asymptotic \eqref{alpha1asym}. So, recall that
\begin{align}\label{zsureq}
    Z^{SU(r)} (q) = \sum_{n \in \Z} \alpha_{1,r}(n) q^n = q^r\sum_{d \mid r} \frac{d}{r^2} \sum_{j=0}^{d-1} \eta\left(\zeta_d^j q^{\frac{r}{d^2}}\right)^{-24}.
\end{align}
Using the substitution $q\mapsto\zeta_d^jq^{\frac{r}{d^2}}$ we see that for any $k\in\Z$, the $k\frac{r}{d^2}$-th coefficient of 
\begin{equation*}
    \eta\left(\zeta_d^j q^{\frac{r}{d^2}}\right)^{-24}
\end{equation*}
is equal to $\zeta_d^{jk}a(k)$. Therefore, the $k\frac{r}{d^2}$-th coefficient of 
\begin{equation}\label{sumetaeq}
    \frac{d}{r^2}\sum_{j=0}^{d-1}\eta\left(\zeta_d^j q^{\frac{r}{d^2}}\right)^{-24}
\end{equation}
is $\frac{d^2}{r^2}a(k)$ if $d\mid k$ and $0$ if $d\nmid k$. In particular, the only non-zero coefficients of $q$ in \eqref{sumetaeq} are the multiples of $r/d$. Hence, the $n^{th}$ coefficient of
\begin{equation*}
    \sum_{d\mid r}\frac{d}{r^2}\sum_{j=0}^{d-1} \eta\left(\zeta_d^j q^{\frac{r}{d^2}}\right)^{-24}
\end{equation*}
is given by
\begin{equation*}
    \sum_{d\mid r,\:\frac{r}{d}\mid n}\frac{d^2}{r^2}a\left(\frac{nd^2}{r}\right).
\end{equation*}
Incorporating the factor of $q^r$ in \eqref{zsureq} and letting $s=r/d$ then gives
\begin{equation}\label{alpha1rconc}
    \alpha_{1,r}(n)=\sum_{s\mid\left(n-r,r\right)}\frac{1}{s^2}a\left(\frac{(n-r)r}{s^2}\right),
\end{equation}
as required. Now, from the asymptotic expression \eqref{aasym1} for $a(n)$ we see that the $s=1$ term dominates in \eqref{alpha1rconc} so that
\begin{equation*}
    \alpha_{1,r}(n)\sim a((n-r)r)\sim a(rn)
\end{equation*}
with each of the error terms in these asymptotics less than $O_r\left(\frac{I_{13}\left(4\pi\sqrt{n}\right)}{e^{2\pi\sqrt{rn}}n^{\frac{13}{2}}}\right)$. This gives \eqref{alpha1asym}.

\subsection{Exact formula for $\alpha_{2,p}(w;n)$}
We now turn to the second Vafa--Witten invariant. To begin with, let $\alpha_{2,p}(w;n)=\alpha_{2,p,1}(w;n)+\alpha_{2,p,2}(w;n)$ where
\begin{align*}
    \sum_{n \in \mathbb{Z}} \alpha_{2,p,1}(w;n) q^{\frac{n}{p}}&=\frac{\delta_{w,0}}{p^2}\eta(q^p)^{-24},\\
    \sum_{n \in \mathbb{Z}} \alpha_{2,p,2}(w;n)q^{\frac{n}{p}}&=\frac{1}{p}\sum_{j=0}^{p-1} e^{-\frac{\pi i j w^2}{p}}\eta(\zeta_p^jq^{\frac{1}{p}})^{-24}.
\end{align*}
Now, the $\frac{n}{p}$th coefficient of $\eta(q^p)^{-24}$ is given by $a(n/p^2)$ if $n\mid p^2$ and $0$ otherwise. Thus,
\begin{equation}\label{w1eq}
    \alpha_{2,p,1}(w;n)=
    \begin{cases}
        \frac{\delta_{w,0}}{p^2}a\left(\frac{n}{p^2}\right)&(n\mid p^2),\\
            0&\text{(otherwise)}.
    \end{cases}
\end{equation}
Then, the $\frac{n}{p}$th coefficient of $\eta(\zeta_p^jq^{\frac{1}{p}})^{-24}$ is given by $\zeta_p^{nj}a(n)$ so that
\begin{equation}\label{w2eq}
    \alpha_{2,p,2}(w;n)=\frac{1}{p}\sum_{j=0}^{p-1}e^{\frac{\pi i j(2n-w^2)}{p}}a(n).
\end{equation}
Combining \eqref{w1eq} and \eqref{w2eq} then gives \eqref{alpha2exact} as desired. Similar to the asymptotics for $\alpha_{1,r}(n)$, the asymptotics for $\alpha_{2,p}(w;n)$ are obtained directly from the asymptotic expression \eqref{aasym1} for $a(n)$. Some care must be taken however, as the sum
\begin{equation*}
    \sum_{j=0}^{p-1}e^{\frac{\pi i j(2n-w^2)}{p}}
\end{equation*}
may vanish depending on the value of $w$, as indicated by \eqref{alpha2asym}. We draw particular attention to the case when $w^2$ is an even integer. Namely, writing $w^2=2m$, we have
\begin{equation}\label{w2new}
    \alpha_{2,p,2}(w;n)=
    \begin{cases}
        a(n)&(p\mid n-m),\\
        0&\text{(otherwise)}.
    \end{cases}
\end{equation}
In addition, if $p\mid m$ we can do even more. In particular, we have the following corollary.

\begin{corollary}\label{alpha2cor}
    Let $\alpha_{2,p}'(w;n)=\alpha_{2,p}(w;pn)$. If $w^2=2m$ is an even integer and $p\mid m$, then
    \begin{equation*}
        Z_w(q)=\sum_{n \in \mathbb{Z}} \alpha_{2,p}'(w;n)q^n.
    \end{equation*}
    That is, all of the non-integer powers of $q$ vanish from the expansion of $Z_w(q)$. Moreover,
    \begin{equation*}
        \alpha_{2,p}'(w;n)=a(pn)+
        \begin{cases}
            \frac{\delta_{w,0}}{p^2}a\left(\frac{n}{p}\right)&(p\mid n)\\
            0&\text{(otherwise)}
        \end{cases}
        .
    \end{equation*}
     As a result, $\alpha_{2,p}'(w;n)\sim a(pn)$ and satisfies the same asymptotics \eqref{alpha1asym} as $\alpha_1(n)$.
\end{corollary}
\begin{proof}
    We simply apply \eqref{w2new}, noting that if $p\mid m$ then $p\mid n-m$ is equivalent to $p\mid n$.
\end{proof}

\subsection{Exact formulae for $\alpha_{3,p}(n)$}
Finally we turn to proving an exact formula for the third Vafa--Witten invariant in question. Let
\begin{align}
    \sum_{n\geq 0}\alpha_{3,p,1}(n)q^{\frac{n}{p}}&=\frac{1}{p^2}q^p\eta\left(q^p\right)^{-24},\label{alpha31}\\
    \sum_{n\geq 0}\alpha_{3,p,2}(n)q^{\frac{n}{p}}&=p^{21}q^p\eta\left(q^{\frac{1}{p}}\right)^{-24},\label{alpha32}\\
    \sum_{n\geq 0}\alpha_{3,p,3}(n)q^{\frac{n}{p}}&=p^{\rho(S)-1}q^p\sum_{j=1}^{p-1}\eta\left(e^{\frac{2\pi i j}{p}}q^{\frac{1}{p}}\right)^{-24}\label{alpha33}
\end{align}
so that $\alpha_{3,p}(n)=\alpha_{3,p,1}(n)+\alpha_{3,p,2}(n)+\alpha_{3,p,3}(n)$. Now, \eqref{alpha31} and \eqref{alpha32} are simple transformations of the function $\eta(q)^{-24}=\sum a(n)q^n$. In particular
\begin{equation*}
    \alpha_{3,p,1}(n)=
    \begin{cases}
        \frac{1}{p^2}a\left(\frac{n}{p^2}-1\right)& ( p^2 \mid n),\\
        0&\text{(otherwise)}
    \end{cases}
\end{equation*}
and
\begin{equation*}
    \alpha_{3,p,2}(n)=p^{21}a(n-p^2).
\end{equation*}
As for \eqref{alpha33}, we note similarly as in Section \ref{alpha1sect} that the $\frac{n}{p}$-th coefficient of 
\begin{equation*}
    \eta\left(e^{\frac{2\pi ij}{p}}q^{\frac{1}{p}}\right)
\end{equation*}
is given by $\zeta_p^{jn}a(n)$. Thus, the $\frac{n}{p}$-th coefficient of
\begin{equation*}
    \sum_{j=1}^{p-1}\eta\left(e^{\frac{2\pi i j}{r}}q^{\frac{1}{p}}\right)^{-24}=-\eta\left(q^{\frac{1}{p}}\right)^{-24}+\sum_{j=0}^{p-1}\eta\left(e^{\frac{2\pi i j}{p}}q^{\frac{1}{p}}\right)^{-24}
\end{equation*}
is $(p-1)a(n)$ if $p\mid n$ or $-a(n)$ if $p\nmid n$. Therefore,
\begin{equation*}
    \alpha_{3,p,3}(n)=
    \begin{cases}
        p^{\rho(S)-1}(p-1)a(n-p^2)&(p\mid n),\\
        -p^{\rho(S)-1}a(n-p^2)&(p\nmid n).
    \end{cases}
\end{equation*}
Combining our expressions for $\alpha_{3,p,1}(n)$, $\alpha_{3,p,2}(n)$ and $\alpha_{3,p,3}(n)$ gives \eqref{alpha3exact} as required. Again, to obtain the asymptotics \eqref{alpha3asym}, we use the asymptotic expression \eqref{aasym1} for $a(n)$.

We also note that the exact formula \eqref{alpha3exact} for $\alpha_{3,p}(n)$ simplifies in the case when the surface $S$ is supersingular (i.e.\@ $\rho(S)=22$). In particular, all non-integral powers of $q$ vanish, leading to the following corollary.

\begin{corollary}\label{alpha3cor}
Let $\alpha_{3,p}'(n)=\alpha_{3,p}(pn)$. Then, if the surface in question is supersingular, we have
\begin{equation*}
    Z_{tw}^{SU(p)/\Z_p}(q)=\sum_{n\geq 0}\alpha_{3,p}'(n)q^n
\end{equation*}
and
\begin{equation*}
        \alpha_{3,p}'(n)= p^{22}a(p(n-p))
        +
        \begin{cases}
            \frac{1}{p^2}a\left(\frac{n}{p}-1\right)& (p\mid n)\\
            0&\text{(otherwise)}
        \end{cases}
        .
\end{equation*}
Furthermore,
    \begin{align*}
        \alpha_{3,p}'(n)&=\frac{2\pi p^{\frac{31}{2}}}{n^{\frac{13}{2}}}I_{13}\left(4\pi\sqrt{pn}\right)\left(1+O_p\left(\frac{1}{e^{2\pi\sqrt{pn}}}\right)\right) \sim \frac{p^{\frac{61}{4}}e^{4\pi\sqrt{pn}}}{\sqrt{2}n^{\frac{27}{4}}}.
    \end{align*}
\end{corollary}
\begin{proof}
    Let $\rho(S)=22$ in \eqref{alpha3exact} and \eqref{alpha3asym}.
\end{proof}

\subsection{Higher order Tur\'{a}n inequalities}\label{Sec: turan}
In this section we prove that all relevant Vafa--Witten invariants asymptotically satisfy the higher-order Tur\'{a}n inequalities. 
\begin{proof}[Proof of Theorem \ref{Thm: main asymp Turan}]
 We have
\begin{align*}
\alpha_{1,r}(n) = \frac{2 \pi}{(nr)^{\frac{13}{2}}} I_{13}(4\pi \sqrt{nr}) \left(1+O_r\left(\frac{1}{e^{2\pi\sqrt{rn}}}\right)\right)
\end{align*}
as $n \to \infty$. This implies the asymptotic formula
\begin{align*}
    \log\left(\frac{\alpha_{1,r}(n+k)}{\alpha_{1,r}(n)}\right) &\sim \log \left( \frac{n^\frac{13}{2} I_{13}(4\pi \sqrt{(n+j)r})}{(n+j)^{\frac{13}{2}} I_{13}(4\pi\sqrt{nr})}\right) \\
    &= 4\pi \sqrt{r} \sum_{i \geq 1} \binom{1/2}{i} \frac{k^i}{n^{\frac{1}{2}-i}} + \frac{21}{4} \sum_{i\geq 1} \frac{(-1)^{i-1} k^i}{in^i} + \sum_{s,t \geq 1} c_t \binom{-t}{s} \frac{ k^s}{n^{s+t}}
\end{align*}
for some (computable) constants $c_t$ arising from the asymptotic expansion for the $I$-Bessel function, valid to all orders of $n^{-\frac{1}{2}}$. We may then apply Theorem \ref{Thm: Gorz} with $A(n)$ and $\delta(n)$ given by
\begin{align*}
    A(n) = 2\pi \sqrt{\frac{r}{n}} +O\left(\frac{1}{n}\right),\qquad \delta(n)^2 = \frac{\pi}{2} \sqrt{r}n^{-\frac{3}{2}} + O \left( n^{-\frac{5}{4}}\right).
\end{align*}
Noting that the asymptotic hyperbolicity of the relevant Jensen polynomials implies that all higher order Tur\'{a}n inequalities are satisfied for large enough $n$, finishing the proof in this case. The arguments for $\alpha_{2,p}$ and $\alpha_{3,p}$ are very similar, where for $\alpha_{2,p}$ one accounts for the variations depending on $w^2$.
\end{proof}

\section{Further results}
We end by indicating related results and their implications to $\alpha_{1,r}(n)$, $\alpha_{2,p}(n)$ and $\alpha_{3,p}(n)$. 
\subsection{Laguerre--Poly\'{a} classes of functions}
A real entire function $\psi(x) = \sum_{k \geq 0} \frac{\gamma_k}{k!}x^k$ is said to lie in the Laguerre--Poly\'{a} class if it can be represented by
\begin{align*}
    \psi(x) = Cx^m e^{bx-ax^2} \prod_{k=1}^r \left(1+\frac{x}{x_k}\right) e^{-\frac{x}{x_k}}, \qquad 0\leq r\leq \infty,
\end{align*}
with $b,C,x_k \in \mathbb{R}$, $m \in \mathbb{Z}_{\geq 0}$, $a\geq 0$, and $\sum_{k=1}^{r} x_k^{-2} < \infty$. In \cite{Wagner}, Wagner introduced new classes of objects known as the shifted-Laguerre--Poly\'{a} classes, and we recall the definitions here.

A real entire function $\psi(x)$ lies in the shifted-Laguerre--Poly\'{a} class of degree $d$ if it is the uniform limit of polynomials $\{\psi_k(x)\}_{k \geq 0}$ with the property that there exists $N(d)$ such that $\psi_n^{(d_n-d)}(x)$ has all real roots for $n \geq N(d)$. The function $\psi$ is then said to lie in the shifted Laguerre--Poly\'{a} class if it satisfies the above for all $d \in \mathbb{N}$.

Most importantly for the present paper is \cite[Theorem 1.1]{Wagner}, which states that if $\{\gamma_k\}_{k\geq 0 }$ is a sequence of eventually non-negative reals, then the following are equivalent\footnote{Here we use the second remark under \cite[Theorem 1.1]{Wagner} to rewrite the theorem in the general setting where we do not assume that the real roots have the same sign.};

\begin{enumerate}[label=\roman*)]
\item The sequence $\{\gamma_k\}$ is a shifted multiplier sequence of type 2 (in the sense of \cite{Wagner}).

    \item For every $d \in \mathbb{N}$ there exist $N(d)$ such that the Jensen polynomial $J_{\gamma}^{d,n}(X)$ has all real roots for $n \geq N(d)$.
    
    \item The formal power series $\sum_{k \geq 0}\frac{\gamma_k}{k!}x^k$ defines a function in the shifted Laguerre--Poly\'{a} class.
\end{enumerate}

In Section \ref{Sec: turan}, we proved that the Jensen polynomial attached to each of the Vafa--Witten invariants $\alpha_{1,r}(n)$, $\alpha_{2,p}(n)$ and $\alpha_{3,p}(pn)$ (throughout replacing $\alpha_{2,p}(n)$ by $\alpha_{2,p}(pn)$ or $\alpha_{2,p}(p^2n)$ as necessary) for fixed $r$ and $p$ tend to the Hermite polynomials as $n \to \infty$, and so have all real roots. Therefore, we immediately obtain the following corollary.

\begin{corollary}\label{Cor: LP}
Let $r$ be a fixed positive integer and $p$ be a fixed prime. Then each sequence $\alpha_{1,r}$, $\alpha_{2,p}$ and $\alpha_{3,p}$ defines a shifted multiplier sequence of type 2, and the formal power series
\begin{align*}
	\sum_{n \geq 0} \frac{\alpha_{1,r}(n)}{n!} x^n, \qquad 	\sum_{n \geq 0} \frac{\alpha_{3,p}(pn)}{(pn)!} x^{pn},
\end{align*}
along with
\begin{align*}
 \sum_{n \geq 0} \frac{\alpha_{2,p}(w;n)}{(n)!} x^{n} \hspace{5pt} \text{ if $\gcd(w^2,2p)=1$}, \qquad	\sum_{n \geq 0} \frac{\alpha_{2,p}(w;pn)}{(pn)!} x^{pn} \hspace{5pt} \text{ if $w^2 =2pm$},
\end{align*}
and
\begin{align*}
  \sum_{n \geq 0} \frac{\alpha_{2,p}(w;p^2n)}{(p^2n)!} x^{p^2n}
\end{align*}
each define a function in the shifted Laguerre--Poly\'{a} class.
\end{corollary}

\subsection{Representations as Poincar\'{e} series and algebraic formulae}
Let $\Gamma_\infty$ be the stabilizer of $i\infty$ in $\Gamma \coloneqq \SL_2(\Z)$ and $\mid_{k,\nu}$ the usual Petersson slash operator with weight $k$ and multiplier $\nu$ (see e.g.\@ \cite{Iw}). For $\kappa \in \frac{1}{2}\Z$ with $\kappa \geq \frac{5}{2}$, and $m \in \frac{1}{24}\Z$ such that $\kappa - 12m \in 2\Z$ the Poincar\'{e} series of weight $\kappa$ and index $m$ that we consider is defined by
\begin{align*}
    P_{k,m}(\tau) \coloneqq \sum_{M \in \Gamma_\infty \backslash \Gamma} q^m \mid_{\kappa,\nu_\eta^{24m}} M,
\end{align*}
with $\nu_\eta$ the multiplier of $\eta$. In \cite{PW}, the authors showed that $a(n)$ is given by the coefficient of $q$ in $P_{14, 1-n}$, and thus we can build linear combinations of weight $14$ Poincar\'{e} series with varying index whose coefficients are precisely $\alpha_{1,r}(n)$, $\alpha_{2,p}(n)$ or $\alpha_{3,p}(n)$. For example, for fixed $r$, we have
\begin{align*}
    \alpha_{1,r}(n) = [\text{coeff. }q] \sum_{s\mid (n-r,r)} \frac{1}{s^2} P_{14,1-\frac{(n-r)r}{s^2}}(\tau).
\end{align*}
Although this formula looks a little messy, it simplifies nicely in explicit examples; taking $r=2$ for instance gives
\begin{align*}
    \alpha_{1,2}(n) = [\text{coeff. }q] \left( P_{14,3-4m}(\tau) + \frac{\delta_{\text{odd}}}{4}  P_{14, 2-m}(\tau) \right),
\end{align*}
where $\delta_{\text{odd}} = 0$ if $n=2m+1$ and $1$ if $n=2m$. Formulae for the remaining Vafa--Witten invariants as coefficients of linear combinations of Poincar\'{e} series follow similarly.

In their famous paper \cite{BO}, Bruinier and Ono proved a formula for the coefficients of $\frac{1}{\eta(q)}$ as a sum of certain algebraic integers. Since each of the Vafa--Witten invariants in this paper have a generating function that is (essentially) a linear combination of $\frac{1}{\eta(q)^{24}}$, one could obtain thus obtain exact algebraic formulae for each of $\alpha_{1,r}(n)$, $\alpha_{2,p}(n)$ and $\alpha_{3,p}(n)$ in a similar fashion.

\end{document}